\documentclass[12pt]{article}

\usepackage{amssymb}
\usepackage{mathrsfs}
\usepackage{bbm}
\usepackage{amsmath}
\usepackage{multirow}
\usepackage{cases} %
\usepackage{graphicx}
\usepackage[dvips]{color}

\renewcommand{\Box}{\framebox{\rule{0.3em}{0.0em}}}

\newtheorem{thm}{Theorem}[section]

\newtheorem{proposition}{Proposition}[section]
\newtheorem{example}{Example}[section]

\newtheorem{defn}{Definition}[section]

\newtheorem{corollary}{Corollary}[section]

\newcommand{\bgeqn}{\begin{eqnarray}}
\newcommand{\edeqn}{\end{eqnarray}}

\newcommand{\bgeq}{\begin{eqnarray*}}
\newcommand{\edeq}{\end{eqnarray*}}
\newcommand{\bgc}{\begin{center}}
\newcommand{\edc}{\end{center}}

\renewcommand{\Box}{\hfill \rule{2.3mm}{2.3mm}}

\makeatletter

\@addtoreset{equation}{section} \makeatother
\newenvironment{proof}{\noindent{\bf Proof. }}{\hfill $\Box$\medskip}

\title{Necessary optimality conditions for optimal control problems with nonsmooth mixed state and control  constraints}
\author{An Li\thanks{\baselineskip 9pt School of Mathematical Sciences, Xiamen University, Xiamen 361005, Fujian, China. The research of this author was partially
supported by the National Natural Science Foundation of China (Grant No. 11101342) and the Fundamental Research Funds for the Central Universities (Grant No. 2012121003, No. 20720150007).}
\ \ and \ \
Jane J. Ye\thanks{Corresponding author. Department of Mathematics and Statistics, University of Victoria, Victoria, B.C., Canada V8W 2Y2,
e-mail: janeye@uvic.ca.  The research of this author was partially supported by NSERC.}
}
\date{}
\begin{document}
\maketitle

\baselineskip 16pt

\vspace{4pt}{\bf Abstract.} \ In this paper we study an optimal control problem with nonsmooth mixed state and control constraints.
In most of the existing results, the necessary optimality condition for optimal control problems with mixed state and control constraints are derived under the Mangasarian-Fromovitz condition and under the assumption that the state and control constraint functions are smooth. In this paper we derive necessary optimality conditions for problems with nonsmooth mixed state and control constraints
under constraint qualifications based on pseudo-Lipschitz continuity and calmness of certain set-valued maps.  The necessary conditions are stratified, in the sense that they are asserted on precisely the domain upon which the hypotheses (and the optimality) are assumed to hold. Moreover  necessary optimality conditions with an Euler inclusion taking an explicit multiplier form are derived for certain cases.

\vspace{4pt}{\bf Key Words.}  \ optimal control, mixed state and control constraint, differential inclusion, necessary optimality
conditions, nonsmooth analysis, calmness.

\vspace{4pt}{\bf AMS subject classification: 45K15, 49K21,49J53}

\medskip

\baselineskip 16pt
\parskip 2pt



\maketitle

\baselineskip 16pt

\medskip

\baselineskip 16pt
\parskip 2pt

\newpage

\section{Introduction.}
Optimal control problems with standard cost and dynamics, but in which the state $x$ and control $u$ are subject to  joint or {\em mixed state and control constraints}  has long been known to be a challenging problem.
Optimality conditions for control problems with mixed state-control constraints have
been the focus of attention for a long time and remain an active research area; see \cite{cr,clarke,clar,cp,pinho1,pinho2,pinho3,pinho7, pinho4,pinho5,pinho6,ledyaev,dmitruk,dub,Hestenes,makowski,Milyutin,Neustadt,zeidan,schwar,Stefani,y01,y02}.
In a recent monograph \cite{cr}, Clarke derived some necessary optimality conditions for a general optimal  differential inclusion problem. These conditions
constitute a new state of the art in optimal control theory, subsuming, unifying, and substantially extending the results in the literature. Using the results from \cite{cr},  Clarke and De Pinho \cite{cp} developed  necessary optimality conditions for  optimal control problems with mixed state and control constraints. Their principal result  \cite[Theorem 2.1]{cp} is  a set of necessary optimality conditions obtained under a geometric hypothesis called the {\em bounded slope condition}  for the optimal control problem with mixed state and control in the form $(x(t), u(t)) \in S(t)$. These theorems  unify and significantly extend most of the existing results. Moreover the necessary conditions are stratified, in the sense that they are asserted on precisely the domain upon which the hypotheses (and the optimality) are assumed to hold.

 In recent years,  the study for optimization problems where the constraints are in the form
$\Phi(x)\in \Omega$, where $\Omega$ is a closed set,  has attracted many attentions. Such a constraint is very general. It includes the classical equality and inequality constraints, the cone constraints,
the variational inequality constraints,
the complementarity constraints and the matrix inequality constraints and is sometimes referred to as a geometric constraint; see   \cite{guolinye-stability,guoyezhang} and the reference within. Very recently, Pang and Stewart \cite{J-SPang} introduced a new class of dynamic problems called {\em differential variational inequalities} (DVIs) which provides a powerful modeling paradigm for many applied problems.  Motivated by the study of DVIs,
in this paper, we consider the   optimal control problem with mixed  state and control constraint in the form
$\Phi(x(t), u(t))\in \Omega (t) $. Such a constraint includes many models of mixed  state and control constraints as a special case, e.g. models with equality and inequality constraints. In particular it can be used to model a DVI.

 The necessary optimality conditions derived in \cite{cp} are all based on  the bounded slope condition and Mangasarian-Fromovitz condition (MFC).
 The  MFC is slightly stronger than the Mansasarian-Fromovitz constraint qualificaiton  (MFCQ) in   mathematical programming. Unfortunately for certain classes of important mathematical programs, MFCQ never holds. Such problems include  the  so-called mathematical program with equilibrium constraints (MPEC) which never satisfy the MFCQ (\cite[Proposition 1.1]{Jane-zhu-zhu-siopt1997}).  In recent years, there has been great progresses towards deriving necessary optimality conditions for mathematical programs under constraint qualifications weaker than MFCQ. To the best of our knowledge, in optimal control theory  all existing works in necessary optimality conditions
 for control problems with constraints require MFC. One of the main purpose of this paper is to fill this gap by deriving  necessary optimality conditions under  constraint qualifications weaker than  MFC.

 Another focus of this paper is to investigate the possibilities of deriving necessary optimality conditions for optimal control problems with nonsmooth mixed state and control constraints, i.e., the mapping $\Phi(x,u)$ is Lipschitz but nonsmooth.  To the best of our knowledge,  there are very few results in this area with the exception of \cite{cp,pinho-silva,ledyaev}.  Moreover we wish to derive a necessary optimality condition with an explicit multiplier for the constraint satisfying the Euler inclusion. The results in \cite{cp} only  show that the Euler inclusion has an explicit multiplier form for the case of smooth inequality and equality constraints, the result in \cite{pinho-silva} is in a weak form and the result in \cite{ledyaev} does not include the Euler inclusion.  In this paper   this issue  is resolved completely in the case of inequality constraints and partially resolved for other cases.

The paper is organized as follows. Section 2 contains preliminaries and preliminary results on variational analysis.
In Section 3, we propose  necessary optimality conditions for a $W^{1,1}$ local minimum of radius $R(\cdot)$ for optimal differential inclusion problems.
In Section 4, we derive  necessary optimality conditions for  a $W^{1,1}$ local minimum of radius $R(\cdot)$ for optimal control problems with a  geometric constraint using the results from Section 3.
In Section 5, we specialize the results from  Section 4 to the case of the classical mixed constraints with equalities and inequalities. We make concluding remarks in Section 6.

\section{Preliminaries and preliminary results}

In this section we present preliminaries and preliminary results on nonsmooth and variational analysis that will be needed in this paper. We give only concise definitions
and conclusions that will be needed in the paper.  For more detailed information on the subject  we refer the reader to  \cite{c,clsw,m2,rw}.

We denote by $B$ and $B(x,\delta)$  the open unit ball and the open ball centered at $x$ with radius $\delta>0$ respectively.
$\bar{\Omega}$ is the closure of a set $\Omega$.
Given a nonempty closed subset $\Omega \subset \Re^n$ and a point $\bar{x}\in \Omega$, we say that $\zeta\in \Re^n$ is a {\em proximal normal vector} to $\Omega$ at $\bar{x}$
if there exists $\sigma=\sigma(\zeta,\bar{x})\geq 0$ such that $\langle\zeta,x-\bar{x}\rangle\leq \sigma|x-\bar{x}|^2, \forall x\in \Omega$,
where $\langle a, b\rangle$ denotes the inner product of vectors $a$ and $b$, $|\cdot|$ denotes the Euclidean norm. The set of such $\zeta$,
denoted by $N^{P}_{\Omega}(\bar{x})$, is termed {\em the proximal normal cone} to $\Omega$.
The {\em limiting normal cone} $N^{L}_{\Omega}(\bar{x})$ to $\Omega$ is defined by
$$N^{L}_{\Omega}(\bar{x}):=\{\lim\zeta_{i}:\zeta_{i}\in N^{P}_{\Omega}(x_{i}),x_{i} \xrightarrow{ \Omega} {\bar{x}}\},$$
where $x_{i} \xrightarrow{ \Omega} {\bar{x}}$ means that $x_i\in \Omega$ and $x_i\rightarrow \bar{x}$.
Consider a lower semicontinuous function $f:\Re^n\rightarrow \Re\cup\{+\infty\}$ and a point $\bar{x}$ where $f$ is finite.
A vector $\zeta\in \Re^n$ is called a {\em proximal subgradient} of $f$ at $\bar{x}$ provided that there exist $\sigma,\delta>0$ such that
$$f(x)\geq f(\bar{x})+\langle\zeta,x-\bar{x}\rangle-\sigma|x-\bar{x}|^{2},\forall x\in B(\bar{x},\delta).$$
The set of such $\zeta$  is denoted $\partial^Pf(\bar{x})$ and referred to as the {\em proximal subdifferential}.
The {\em limiting subdifferential} of $f$ at $\bar{x}$  is the set
$$\partial^Lf(\bar{x}):=\{\lim\zeta_{i}:\zeta_{i}\in \partial^Pf(x_{i}),x_{i}\rightarrow \bar{x}, f(x_{i})\rightarrow f(\bar{x})\}.$$
For a locally Lipschitz function $f$ on $\Re^n$,  the generalized gradient $\partial^Cf(\bar{x})$ coincides with
$co\partial^Lf(\bar{x})$, the convex hull of $\partial^Lf(\bar{x})$; further the associated Clarke normal cone
$N^{C}_{\Omega}(\bar{x})$ at $\bar{x}\in \Omega$ coincides with $\overline{co}N^{L}_{\Omega}(\bar{x})$, the closure of the convex hull of
$N^{L}_{\Omega}(\bar{x})$.

Let  $\Psi: \Re^n\rightrightarrows\Re^q$ be an arbitrary set-valued map (assigning to each $z\in \Re^n$, a set
$\Psi(z)\subset \Re^q$ which may be  empty).  The graph  of $\Psi$ is denoted by  $gph\Psi$ (i.e., $gph \Psi:=\{(z,v): v\in \Psi(z)\}$). Let  $(\bar{z},\bar{v})\in \overline{gph \Psi}$. The set-valued map $D^*\Psi(\bar{z},\bar{v})$
from $\Re^q$ into $\Re^n$ defined by
$$D^*\Psi(\bar{z},\bar{v})(\eta)=\{\xi\in \Re^n:(\xi,-\eta)\in N_{gph\Psi}^{L}(\bar{z},\bar{v})\}$$
is called the coderivative of $\Psi$ at the point $(\bar{z},\bar{v})$. The symbol $D^*\Psi(\bar{z})$ is used when $\Psi$ is single-valued
at $\bar{z}$ and $\bar{v}=\Psi(\bar{z})$. In particular if $\Psi: \Re^n\rightarrow \Re^q$ is single-valued and Lipschitz near $\bar{z}$, then
$$D^*\Psi(\bar{z})(\eta)=\partial^L\langle\eta, \Psi \rangle (\bar{z})\quad \forall \eta \in \Re^q.$$

We now review some concepts of Lipschitz continuity of set-valued mappings.
\begin{defn}\cite{robin}
A set-valued map $\Psi: \Re^n\rightrightarrows \Re^q$ is said to be upper-Lipschitz at $\bar{z}$ if
there exist $\mu>0$ and a neighborhood $U$ of $\bar{z}$ such that
$$\Psi(z)\subset \Psi(\bar{z})+\mu|z-\bar{z}|\bar{B},  \,\,\forall z\in U.$$
\end{defn}
\begin{defn}\cite{aubin}\cite[Definition 1.40]{m2} \label{def2.2}
A set-valued map $\Psi: \Re^n\rightrightarrows \Re^q$ is said to be pseudo-Lipschitz (or locally Lipschitz like) around $(\bar{z},\bar{v})\in gph\Psi$ if there exist
 $\mu>0$ and a neighborhood $U$ of $\bar{z}$, a neighborhood $V$ of $\bar{v}$ such that
$$\Psi(z)\cap V\subset \Psi({z'})+\mu|z-{z'}|\bar{B},  \,\,\forall z,z'\in U.$$
The number $\mu$ satisfying the above inclusion is called a modulus and the infimum of all such moduli $\{\mu\}$ is called the exact Lipschitz bound of $\Psi$ around $(\bar{z},\bar{v})$ and is denoted by ${\rm lip}\Psi(\bar{z},\bar{v})$.
\end{defn}
\begin{defn} \cite{y05,rw}
A set-valued map $\Psi: \Re^n\rightrightarrows \Re^q$ is said to be calm at $(\bar{z},\bar{v})\in gph\Psi$ if
there exist $\mu>0$ and a neighborhood $U$ of $\bar{z}$, a neighborhood $V$ of $\bar{v}$ such that
$$\Psi(z)\cap V\subset \Psi(\bar{z})+\mu|z-\bar{z}|\bar{B},  \,\,\forall z\in U.$$
\end{defn}
Calmness is weaker than both
Aubin's pseudo-Lipschitz continuity \cite{aubin} and Robinson's upper-Lipschitz continuity \cite{robin}.
For recent discussion on the properties and the criterion of calmness of a set-valued map,
see (\cite{hen,adi}).

The following condition characterizes the pseudo-Lipschitz continuity of a set-valued map.
\begin{proposition}\label{prop2.0}\cite[Theorem 4.10 and Equation 1.22]{m2} Let  $\Psi: \Re^n\rightrightarrows \Re^q$ be a set-valued map with a closed graph. $\Psi$ is  pseudo-Lipschitz around $(\bar{z},\bar{v})\in gph\Psi$ if and only if one of the following conditions hold:
\begin{itemize}
\item[{\rm (a)}]  $D^* \Psi(\bar{z},\bar{v})(0)=\{0\}$;
\item[{\rm (b)}]  $(\alpha,\beta) \in N_{gph\Psi}^L(\bar z,\bar v)\Longrightarrow |\alpha| \leq \mu |\beta|$ for some $\mu>0$;
\item[{\rm (c)}] $\exists \varepsilon>0$ such that $(\alpha,\beta) \in N_{gph\Psi}^P(z, v), \  z\in B(\bar z, \varepsilon), \ v\in B(\bar v, \varepsilon) \Longrightarrow |\alpha| \leq \mu |\beta|$ for some $\mu>0$.
\end{itemize}
Moreover $\mu$ can be taken as ${\rm lip}\Psi(\bar{z},\bar{v})$, the exact Lipschitz bound of $\Psi$ around $(\bar{z},\bar{v})$.
\end{proposition}

 Consider a constrained system defined by a locally Lipschitz map $\Phi(x,u): \Re^n\times \Re^m\rightarrow \Re^d$
 and closed sets $U\subset \Re^m, \Omega\subset \Re^d$:
 $$\{(x,u): u \in U, \Phi(x,u) \in \Omega\}.$$
 Define a set-valued map as  the perturbed constrained system:
\begin{equation}
M(y):=\{(x,u): u\in U, \Phi(x,u) +y \in \Omega\}\label{perturbedmap}.
\end{equation}
For convenience, we summarize the well known criteria for the pseudo-Lipschitz continuity, upper-Lipschitz continuity and the calmness for the set-valued map $M$ in the following proposition.
\begin{proposition}\label{pseudoL} \cite{m2,robin1}
Let $(\bar{x},\bar{u}) \in M(0)$. Suppose that the no nonzero abnormal multiplier constraint qualification (NNAMCQ) holds at $(\bar x, \bar u)$:
\begin{eqnarray*}
\left \{ \begin{array}{l}
(0,0)\in \partial^L\langle \lambda, \Phi(\cdot,\cdot)\rangle(\bar{x},\bar{u})+\{0\}\times N_{U}^{L}(\bar{u})\\
\lambda\in N_{\Omega}^{L}(\Phi(\bar{x},\bar{u}))
\end{array} \right.  && \Longrightarrow\lambda=0.
\end{eqnarray*}Then the   set-valued  map $M$ defined as in (\ref{perturbedmap}) is
pseudo-Lipschitz continuous at $(0,\bar{x},\bar{u})\in gph M$ and hence is calm at $(0,\bar{x},\bar{u})\in gph M$.
If the mapping $\Phi$ is affine and the sets $U$ and $\Omega$ are unions of finitely many polyhedral convex sets, then the set-valued map $M$ is upper-Lipschitz continuous  at any $y\in \Re^d$ and hence is calm at  every point $(y, x,u)\in gph M$.
\end{proposition}
\begin{proof}
The criterion for the pseudo-Lipschitz continuity is well known and can be found, for example in \cite{m2}.
The  criterion for the calmness  follows from the fact that the multifunction $M$ is a polyhedral multifunction and hence locally upper Lipschitz continuous at any point $y$  by the results of Robinson \cite{robin1}.
\end{proof}

Note that the NNAMCQ  is weaker than the generalized Mangasarian Fromovitz Constraint Qualification (GMFCQ)  but equivalent  if  the interior of the Clarke tangent cone to set $U$ is nonempty and $\Omega $ is a convex cone see e.g. \cite{con,y01new}.


Based on   condition (A) in Proposition \ref{prop2.0},  we can obtain a sufficient condition for pseudo-Lipschitz continuity of a set-valued map that has a special structure which will be useful later in our analysis. In fact this set-valued map is the autonomous form of the dynamics in the optimal control problem $(P_c)$ in Section 4. Note that this result is of independent interest.

\begin{proposition} \label{prop2.2} Let the set-valued map $\Gamma:\Re^n\rightrightarrows \Re^q$  be defined by
\begin{equation}
\Gamma(x):=\{\phi(x,u): u\in U, \,\Phi(x,u) \in \Omega\}\label{set-valued map}
\end{equation}
where $\phi:\Re^n\times \Re^m\rightarrow \Re^q$, $\Phi: \Re^n\times \Re^m\rightarrow \Re^d$ are locally Lipschitz continuous
 and $U\subset \Re^m, \Omega\subset \Re^d$ are closed.  Suppose that the  set-valued  map $M$ defined by (\ref{perturbedmap})
is calm at $(0,\bar{x}, \bar{u}) \in gph M$. Then $\Gamma$
is pseudo-Lipschitz around $(\bar{x},\phi(\bar{x}, \bar{u}))\in gph\Gamma$
 if the weak basic constraint qualification (WBCQ) holds at $(\bar x, \bar u)$:
\begin{eqnarray}
\left\{ \begin{array}{l} (\alpha,0)\in \partial^L\langle \lambda, \Phi(\cdot,\cdot)\rangle(\bar{x},\bar{u})+\{0\}\times N_{U}^{L}(\bar{u})\\
\lambda\in N_{\Omega}^{L}(\Phi(\bar{x},\bar{u}) )
\end{array} \right. \qquad \Longrightarrow\alpha=0. \label{cq1section2}
\end{eqnarray}
\end{proposition}
\begin{proof}
By  Proposition \ref{prop2.0}(a), it suffices to verify that $D^* \Gamma(\bar{x},\phi(\bar{x}, \bar{u}))(0)=\{0\}$. By definition,
$\alpha \in D^* \Gamma(\bar{x},\phi(\bar{x}, \bar{u}))(0)$ is equivalent to  $(\alpha,0)\in N_{gph\Gamma}^{L}(\bar{x},\phi(\bar{x}, \bar{u}))$.
Let $(\alpha,0)\in N_{gph\Gamma}^{L}(\bar{x},\phi(\bar{x}, \bar{u}))$.
Then there exist $(x_{i},v_i)\xrightarrow{gph\Gamma}(\bar{x},\phi(\bar{x}, \bar{u}))$ and
$(\alpha_{i},\beta_{i})\rightarrow(\alpha,0)$ such that $(\alpha_{i},\beta_{i})\in N_{gph\Gamma}^{P}(x_{i},v_i)$.
 By the definition of a proximal normal vector, for each
$i$ there exists $\sigma_{i}\geq 0$ such that
$$\langle(\alpha_{i},\beta_{i}),(x,v)-(x_{i},v_i)\rangle\leq \sigma_{i}|(x,v)-(x_{i},v_i)|^{2} \quad \forall (x,v)\in gph \Gamma.$$
Since $v_i \in \Gamma(x_i)$ and $v\in \Gamma(x)$, the above implies that
 if $(x,u)\in M(0)$ we have
$$\langle(\alpha_{i},\beta_{i}),(x,\phi(x,u))-(x_{i},\phi(x_{i},u_{i}))\rangle\leq \sigma_{i}|(x,\phi(x,u))-(x_{i},\phi(x_{i},u_{i}))|^{2}.$$
It follows that the function
$$\Lambda_{i}(x,u):=\langle-(\alpha_{i},\beta_{i}),(x,\phi(x,u))\rangle+ \sigma_{i}|(x,\phi(x,u))-(x_{i},\phi(x_{i},u_{i}))|^{2}$$
has a  minimum at $(x_{i},u_{i})\in M(0)$.
Therefore by the local Lipschitz continuity of the function $\phi$, we have
$$(0,0)\in -(\alpha_i,0)+\partial^L\langle -\beta_i,\phi(\cdot,\cdot)\rangle(x_{i},u_{i})+N_{M(0)}^{L}(x_{i},u_{i}).$$
Passing to the limit, by virtue of the local Lipschitz continuity of $\phi$ and the closedness of the limiting normal cone mapping, we get  that
$(\alpha,0)\in N_{M(0)}^{L}(\bar{x}, \bar{u})$.
Since the set-value map $M(y)$ is calm at $(0,\bar{x}, \bar{u})$, we get that, by \cite[Proposition 3.4]{adi}
 there exists
$\lambda\in N_{\Omega}^{L}(\Phi(\bar{x}, \bar{u}))$ such that $(\alpha,0)\in D^{*}\Phi(\bar{x}, \bar{u})(\lambda)+\{0\}\times N_{U}^{L}(\bar{u})$.
 Since $D^{*}\Phi(\bar{x}, \bar{u})(\lambda)=\partial^L\langle \lambda, \Phi(\cdot,\cdot)\rangle(\bar{x}, \bar{u})$ by virtue of  \cite[Proposition 2.11]{morduk}, we get
$$(\alpha,0)\in \partial^L\langle \lambda, \Phi(\cdot,\cdot)\rangle(\bar{x}, \bar{u})+\{0\}\times N_{U}^{L}(\bar{u}).$$
Therefore, by the constraint qualification, we get $\alpha=0$. The proof is therefore complete.
\end{proof}

Following \cite{cp}, we say that Mangasarian-Fromovitz condition (MFC) holds at $(\bar x,\bar u)$ if
\begin{eqnarray}
\left\{ \begin{array}{l} (\alpha,0)\in \partial^L\langle \lambda, \Phi(\cdot,\cdot)\rangle(\bar{x},\bar{u})+\{0\}\times N_{U}^{L}(\bar{u})\\
\lambda\in N_{\Omega}^{L}(\Phi(\bar{x},\bar{u}) )
\end{array} \right. \qquad \Longrightarrow\lambda=0. \label{cq3section2}
\end{eqnarray}
According to \cite{cp}, we say that a calibrated constraint qualification (CCQ) holds at $(\bar x,\bar u)$ if there exists $\mu>0$ such that
\begin{eqnarray*}
\left\{ \begin{array}{l} (\alpha ,\beta)\in \partial^L\langle \lambda, \Phi(\cdot,\cdot)\rangle(\bar{x},\bar{u})+\{0\}\times N_{U}^{L}(\bar{u})\\
\lambda\in N_{\Omega}^{L}(\Phi(\bar{x},\bar{u}) )
\end{array} \right. \qquad \Longrightarrow |\lambda |\leq \mu |\beta|.
\end{eqnarray*}
It is easy to check that the following implications hold:
$$\mbox{CCQ} \Longrightarrow \mbox{MFC} \Longleftrightarrow \mbox{WBCQ+NNAMCQ}  \Longrightarrow \mbox{WBCQ + Calmness of $M$}.$$

We now  give some sufficient conditions under which the set-valued map $M$ is calm at $(0,\bar x,\bar u)$.
It is easy to check that the calmness condition of $M$ at $(0,\bar x, \bar u)$  holds if and only if  the set
$$M(0)=\{(x,u): u\in U, \Phi(x,u) \in \Omega\}$$
has a local error bound at $(\bar x,\bar u)$, i.e., there exists  nonnegative constant $\mu$ and $\delta>0$ and such that
$$dist_{M(0)} (x,u)\leq \mu dist_{\Omega} (\Phi(x,u))\quad \forall (x,u) \in B((\bar x, \bar u),\delta), u\in      U,$$
where $dist_\Omega(z)$ denotes the distance from $z$ to the set $\Omega$.
In particular when $\Omega=\Re^l_+\times \Re^s$ and $\Phi=(g,h)^T$, $M$ is calm at $(0,\bar x,\bar u)$ if and only if
 there exists a nonnegative constant $\mu$ and $\delta>0$ such that
\begin{eqnarray*}
dist_{M(0)} (x,u) &\leq& \mu (\max\{g_1(x,u),\dots,g_l(x,u),0\}+|h(x,u)|)\quad \\
&& \qquad \qquad \forall (x,u) \in B((\bar x, \bar u),\delta), u\in U.
\end{eqnarray*}
There are many conditions under which the local error bound holds (see e.g. Wu and Ye \cite{Wu-Ye01,Wu-Yemp,Wu-Ye03}). For example it holds when $g$ and $h$ are affine and the set $U$ is a union of polyhedral convex sets,
or if the $(\bar x,\bar u)$ is quasinormal (see \cite{guoyezhang}). We recall the following results for the existence of a global error bound (i.e. when $\delta=\infty$ in the error bound condition).
\begin{proposition}\label{prop2.4}\cite[Corollary 4.8]{Wu-Ye03} Let $z=(x,u)\in \Re^{n+m}$, $g_i(z):=\varphi_i(z)+b_i^T z, \ i=1,\dots, l$ where
$b_i=(b_{i1},\dots, b_{i(n+m)})\in \Re^{n+m}$. Suppose that there exists some $j\in \{1,\dots, n+m\}$ such that $b_{ij}, i=1,\dots, l$ have the same sign and $\varphi_i$ is a differentiable function that is independent of the $j$ th component of vector $z\in \Re^{n+m}$. Then
$$M(0):=\{z: g_i(z)\leq 0 \ i=1,\dots, l\}$$
has a global error bound for all $z\in \Re^{n+m}$.
\end{proposition}

In the special case when $U$ is the whole space $\Re^m$ and the functions $g,h$ are continuously differentiable, some of the recent developments in constraint qualifications
for nonlinear programming lead to  verifiable sufficient conditions for the existence of a local error bound. We now summarize these constraint qualifications. Denote by $z:=(x,u)$ and
\begin{equation}
M(0):=\{z\in \Re^{n+m}: g(z)\leq 0, h(z)=0\}\label{Mzero}
\end{equation}
where $g:\Re^{n+m}\rightarrow \Re^l, h:\Re^{n+m}\rightarrow \Re^s$ are continuous differentiable.

Recall that, for ${\cal A}:=\{a^1,\cdots,a^l\}$ and ${\cal B}:=\{b^1,\cdots,b^s\}$,
$({\cal A},{\cal B})$ is said to be {\it positively linearly dependent} iff there exist
$\alpha$ and $\beta$ such that $\alpha\geq0$, $(\alpha,\beta)\neq0$, and
$$\sum_{i=1}^l\alpha_ia^i+\sum_{j=1}^s\beta_jb^j=0.$$ Otherwise, $({\cal A},{\cal B})$ is said to be  {\it positively linearly independent}.
\begin{defn}\rm\label{mfcq---}
 Let $z_*\in M(0)$ defined as in (\ref{Mzero}). Denote the active set of $g$ at $z_* $ by $I_g^*:=\{i: g_i(z_*)=0\}$.
\begin{itemize}
\item[{\rm(i)}] We say that the {\it positive linear
independence constraint qualification} (PLICQ) holds at $z_*$ iff the family of
gradients $(\{\nabla g_i(z_*)\, | \ i\in {I}_g^*\},\{\nabla h_j(z_*)\, | \
j=1,\cdots,s\})$ is positively linearly independent.

\item[{\rm(ii)}] We say that the {\it constant positive linear
dependence} condition (CPLD) holds at $z_*$ iff, for each ${\cal I}\subseteq
{I}_g^*$ and ${\cal J}\subseteq \{1,\dots, s\}$, whenever $(\{\nabla g_i(z_*)\,
| \ i\in{\cal I}\},\{\nabla h_j(z_*)\, | \ j\in{\cal J}\})$ is positively
linearly dependent, there exists $\delta>0$ such that, for every $z\in {
B}(z_*,\delta)$, $\{\nabla g_i(z), \nabla h_j(z)\, | \ i\in {\cal I}, j\in
{\cal J}\}$ is linearly dependent.

\item[{\rm(iii)}]  Let ${\cal J}\subseteq\{1,\cdots,s\}$
be such that $\{\nabla h_j(z_*)\}_{j\in{\cal J}}$ is a basis for ${\rm
span}\,\{\nabla h_j(z_*)\}_{j=1}^s$. We say  that the {\em relaxed constant
positive linear dependence} condition (RCPLD) holds at $z_*$ iff there exists
$\delta>0$ such that
\begin{itemize}
\item $\{\nabla h_j(z)\}_{j=1}^s$ has the same rank for each $z\in {
B}(z_*,\delta)$;
\item for each ${\cal I}\subseteq {I}_g^*$, if $(\{\nabla g_i(z_*)\, | \ i\in{\cal
I}\},\{\nabla h_j(z_*)\, | \ j\in{\cal J}\})$ is positively linearly dependent, then $\{\nabla
g_i(z), \nabla h_j(z)\, | \ i\in {\cal I}, j\in {\cal J}\}$ is linearly dependent
for each $z\in B(z_*,\delta)$.
\end{itemize}
\item[{\rm(iv)}] Let ${\cal J}_-:=\{i\in I_g^*\,|\ -\nabla g_i(z_*)\in {\cal L}^o(z_*)
\}$, where
$$ {\cal L}^o(z_*):=\{d: d=\sum_{i\in I_g^*} \lambda_i \nabla g_i(z_*)+\sum_{j=1}^s \mu_j \nabla h_j(z_*), \lambda_j \geq 0 \}$$ We say that the {\em
constant rank of the subspace component} condition (CRSC) holds at $z_*$ iff
there exists $\delta>0$ such that the family of gradients $$\{\nabla
g_i(z),\nabla h_j(z) \,|\ i\in {\cal J}_-,j\in\{1,\cdots,s\}\}$$ has the same
rank for every $z\in B (z_*,\delta)$.
\end{itemize}
\end{defn}

The PLICQ is the same as NNAMCQ which is equivalent to the classical MFCQ in this case. The CPLD
was introduced by Qi and Wei in \cite{Qi-Wei} and was shown to be a constraint
qualification by Andreani et al. in \cite{Andreani-cpldcq}. The relaxed CPLD was
introduced by Andreani et al.  in \cite{abdreabi-haeser--}. The CRSC was
introduced by Andreani et al. in \cite{abdeabi--twoCQ}, and is shown to be weaker than the
RCPLD (see \cite[Theorem 4.2]{abdeabi--twoCQ}).
The geometric introduction to these constraint qualifications is given in a recent paper \cite{abdeabi--silva}. We summarize the relationships of these constraint qualifications below.
$$\mbox{PLICQ} \Longrightarrow \mbox{CPLD} \Longrightarrow \mbox{RCPLD}  \Longrightarrow \mbox{CRSC}\Longrightarrow \mbox{ Local Error bound}.$$

Note that WBCQ plus calmness of $M$ is a weaker condition than  NNAMCQ let alone the stronger condition calibrated constraint qualification. To see this point we  consider  the simple case where $\Omega =\{0\}$,  $U=\Re^m$ and $\Phi$ is affine,
NNAMCQ means that the Jacobian matrix $\nabla \Phi(\bar x,\bar u)$ has maximum row rank. The calmness of $M$ satisfies automatically. The WBCQ becomes
$$(\alpha, 0)=\nabla \Phi(\bar x,\bar u)^T\lambda \Longrightarrow  \alpha =0.$$
This means that the row vectors of $\nabla \Phi(\bar x,\bar u)$ may be linearly dependent and hence the Jacobian matrix $\nabla \Phi(\bar x,\bar u)$ does not required to have maximum row rank. More precisely we have the following results.
\begin{proposition} Suppose that $\Gamma(x):=\{\phi(x,u): u\in \Re^m,\Phi(x,u)=0\}$ where $\Phi:\Re^n\times \Re^m\rightarrow \Re^d$ is affine and $\phi:\Re^n\times \Re^m\rightarrow \Re^q$ is locally Lipschitz.  Suppose  that
\begin{eqnarray*}
\left\{ \begin{array}{l} \alpha=\nabla_x \Phi(\bar{x},\bar{u})^T \lambda\\
0=\nabla_u \Phi(\bar{x},\bar{u})^T \lambda , \quad \lambda \in \Re^m
\end{array} \right. \qquad \Longrightarrow \alpha=0.
\end{eqnarray*}
Then WBCQ plus calmness of $M$ holds and hence $\Gamma$ is pseudo-Lipschitz around $(\bar{x},\phi(\bar{x}, \bar{u}))\in gph\Gamma$.
\end{proposition}
\begin{example}
For example, $\Phi_1=x+u_1-u_2$, $\Phi_2=2x+2u_1-2u_2$. The gradients $\nabla\Phi_1=(1,1,-1)$ and $\nabla\Phi_2=(2,2,-2)$  are linearly dependent and hence NNAMCQ fails. We now show that WBCQ plus calmness condition holds.  Since the mapping $\Phi$ is affine, the calmness condition holds. The only $\lambda_1,\lambda_2$ making
$0=\lambda_1 \nabla_u\Phi_1+\lambda_2\nabla_u \Phi_2$
are those satisfying $\lambda_1=-2 \lambda_2$.  With these numbers, we must have
$0=\lambda_1 \nabla_x\Phi_1+\lambda_2\nabla_x \Phi_2.$ Hence the WBCQ holds. Therefore the set-valued map defined by
$$\Gamma(x):=\{\phi(x,u): (u_1,u_2)\in \Re^2,x+u_1-u_2=0, 2x+2u_1-2u_2=0\}$$
is  pseudo-Lipschitz around $(\bar{x},\phi(\bar{x}, \bar{u}))\in gph\Gamma$ for any $(\bar{x}, \bar{u}) \in \Re\times \Re^2$.
\end{example}
\section{Optimal  differential inclusion problem}

In this section we consider the optimal differential inclusion problem:
\begin{eqnarray*}
(P_I)~~~~~\min &&  f(x(t_0),x(t_1))\\
s.t. && \dot{x}(t)\in \Gamma_{t}(x(t)) \,\quad a.e. \,t\in [t_0,t_1]\\
&& (x(t_0),x(t_1))\in E.
\end{eqnarray*}
We assume that  $f:\Re^n\times \Re^n\rightarrow \Re$ is a locally Lipschitz function and $E$ is a closed subset of $\Re^n\times \Re^n$.
$\Gamma_{t}$ is a set-valued map from $\Re^n$ to the subsets of $\Re^n$ for each $t$,
the map $(t,x)\rightarrow \Gamma_{t}(x)$ is ${\cal L}\times {\cal B}$-measurable, where ${\cal L}\times {\cal B}$ denotes the $\sigma$-algebra of subsets of $[t_0,t_1]\times \Re^n$ generated by product
sets $M\times N$ where $M$ is a Lebesgue measurable subset of $[t_0,t_1]$ and $N$ is a Borel subset of $\Re^n$ and the set $G(t):=gph\Gamma_{t}(\cdot)$
is closed for each $t$.
Let $R(\cdot)$ be a measurable function on $[t_0,t_1]$ with values in $(0,+\infty]$. For brevity, we refer to any absolutely continuous function $x: [t_0,t_1]\rightarrow\Re^n$ as an arc. An arc $x_*$ is feasible if it satisfies all constraints.
We say that a feasible arc $x_*$ is a $W^{1,1}$ local minimum of radius $R(\cdot)$ for problem $(P_I)$ if for some $\varepsilon>0$, for every other feasible arc $x$ satisfying
$$|\dot{x}(t)-\dot{x}_{*}(t)|\leq R(t)\,\, a.e. \,t\in [t_0,t_1],  \int_{t_0}^{t_1}|\dot{x}(t)-\dot{x}_{*}(t)|dt\leq \varepsilon,
\| x-x_{*}\|_{\infty}\leq \varepsilon,$$
where $\|\cdot\|_{\infty}$ denotes the  essential  norm in $L^{\infty}$ space,
one has $$ f(x(t_0),x(t_1))\geq f(x_*(t_0),x_*(t_1)).$$
\begin{defn}
\label{def1}
$\Gamma_t$  is said to satisfy a pseudo-Lipschitz condition of radius $R(\cdot)$ near $x_*$ if there exists $\varepsilon>0$ and an uniformly positive integrable function $k:[t_0,t_1]\rightarrow \Re$ (namely,  $k(t)\geq k_{0}>0, \mbox{ a.e.} t\in [t_0,t_1]$ for some constant $k_0$) such that,
for almost all $t\in [t_0,t_1]$,
$$x,x'\in B(x_{*}(t),\varepsilon) \Longrightarrow \Gamma_{t}(x)\cap{ B}(\dot{x}_*(t),R(t))\subset  \Gamma_{t}(x')+k(t)|x-x'|\bar{B}.$$
\end{defn}
 Since the modulus $k(t)$ is required to be integrable and  $R(t)$ is prescribed, the pseudo-Lipschitz condition of radius $R(\cdot)$ implies that for almost all $t\in [t_0,t_1]$, the set-valued map $\Gamma_t(\cdot)$ is  pseudo-Lipschitz around
$(x_*(t),\dot{x}_*(t))\in gph \Gamma_t$ but not vice versa.
\begin{defn}\label{def2}
$\Gamma_t$  is said to satisfy the  tempered growth condition for radius $R(\cdot)$ near $x_*$ if there exist $\varepsilon>0$, $\lambda\in (0,1)$ and an uniformly positive integrable function $r:[t_0,t_1]\rightarrow \Re$ such that for almost every $t\in [t_0,t_1]$,  we have $0<r_0\leq r_0(t)\leq \lambda R(t), \mbox{ a.e. } t$  for some $r_0$ and
$$| x-x_{*}(t)|\leq \varepsilon \Longrightarrow \Gamma_{t}(x)\cap {B}(\dot{x}_*(t),r_0(t)) \not =\emptyset.$$
\end{defn}
Note that our definitions of the pseudo-Lipschitz continuity and the tempered growth condition of radius $R(\cdot)$ are slightly different with the original definitions in \cite{cr} in that the closed ball  in the  definitions of the pseudo-Lipschitz continuity is replaced by the open ball, and the functions $k(t), r(t)$ are required to be bounded away from zero almost everywhere. These kinds of modifications were first introduced in \cite{b}.  The justification of requiring the functions $k(t), r(t)$ to be uniformly
positive integrable is that the set-valued map in the  transformed differential inclusion in \cite{cr}   needs to be bounded-valued.

The following theorem follows from  applying Clarke \cite[Theorem 3.1.1]{cr}. The conclusion is slightly different  with the original  theorem \cite[Theorem 3.1.1]{cr} in that  the Weierstrass condition holds only on the the open ball $B(\dot{x}_*(t),R(t))$. Note that it was shown in \cite{b} through a counter example that  the Weierstrass condition may not hold  on the the closed  ball $\bar{B}(\dot{x}_*(t),R(t))$.

\begin{thm} \label{thm2.1}
Assume that $x_*$  is a local minimum of radius $R(\cdot)$ of the optimal differential inclusion problem $(P_I)$.
Suppose that, for the radius $R(\cdot)$,  $\Gamma_t$ satisfies pseudo-Lipschitz and tempered growth conditions near $x_*$.
Then there exist an arc $p$ and a number $\lambda_0 $ in  $\{0,1\}$ satisfying the nontriviality condition
\begin{equation}
(\lambda_0, p(t))\not =0 \quad \forall t\in [t_0,t_1]\label{nontrivi}
\end{equation}  and the transversality condition:
\begin{equation} (p(t_0),-p(t_1)) \in \lambda_0 \partial^Lf(x_*(t_0),x_*(t_1))+ N^{L}_{E}(x_*(t_0), x_*(t_1)),
\label{transv}
\end{equation}
and such that $p$ satisfies the Euler inclusion
\begin{equation}
 \dot{p}(t)\in co\left\{\omega: (\omega,p(t))\in N_{G(t)}^{L}(x_*(t),\dot{x}_{*}(t))\right\},\, a.e.\, t\in [t_0,t_1],
\label{Euler}
\end{equation}
as well as  the Weierstrass condition of radius $R(\cdot)$:
\begin{equation}
\langle p(t),v-\dot{x}_{*}(t) \rangle\leq 0, \,\,\,\,\, \forall v\in \Gamma_{t}(x_*(t))\cap B(\dot{x}_*(t),R(t)),\,a.e.\, t\in [t_0,t_1].
\label{Weires}
\end{equation}
\end{thm}

Clarke \cite[Corollary 3.5.3]{cr} showed that the following bounded slope condition together with the strengthened tempered growth condition imply the pseudo-Lipschitz and tempered growth conditions of radius $R(\cdot)$ and hence the necessary conditions of Theorem \ref{thm2.1} hold as stated.
\begin{corollary}\label{thm2}
Suppose that in Theorem \ref{thm2.1}, the pseudo-Lipschitz hypothesis is replaced by the following bounded slope condition: there exist  $\varepsilon>0$ and a
 uniformly positive integrable function $k$ such that, for almost every $t$, for every $(x,v) \in G(t)$ with $x\in B(x_*(t),\varepsilon)$ and $v \in {B}(\dot{x}_*(t),R(t))$, for all $(\alpha,\beta) \in N_{G(t)}^P(x,v)$, one has $|\alpha| \leq k(t) |\beta|$. Suppose also that the strengthened  tempered growth condition holds:
$$\hbox{\rm ess inf} \left\{\frac{R(t)}{k(t)} : t\in [t_0,t_1]\right\}>0,$$
where $ess \ inf$ is the essential infimum.
Then the necessary conditions of Theorem \ref{thm2.1} hold as stated, for the same radius $R(\cdot)$.
\end{corollary}

Observe that in both Theorem \ref{thm2.1} and Corollary \ref{thm2}, one requires to verify that the Lipschitz modulus $k(t)$ is  integrable and the tempered growth condition holds. These properties may not be easy to verify. In the rest of this section we will find conditions under which the modulus $k$ is a positive constant and hence integrable automatically.
 For this purpose we first derive the following result.
 Let
$$C_*^{\varepsilon, R}:=cl \{(t,x,v)\in [t_0,t_1]\times \Re^n\times \Re^n: (x,v) \in G(t)\cap  ( \bar B(x_*(t),\varepsilon)\times \bar B(\dot{x}_*(t), R(t)))\},$$
where $cl$ denotes the closure.
\begin{proposition}\label{propositionnew}Assume that $\Gamma_t(x)\equiv\Gamma(x)$ is independent of $t$ and denote by $G\equiv G(t)$.
Suppose that $C_*^{\varepsilon, R}$ is compact and that  for all  $(t,x,v)\in C_*^{\varepsilon, R}$, $D^*\Gamma(x,v)(0)=\{0\}$, then there exists a constant $k>0$ such that for almost every $t$, for every $(x,v) \in G\cap  (  B(x_*(t),\varepsilon) \times B(\dot{x}_*(t), R(t)))$, the bounded slope condition holds:
$$(\alpha,\beta) \in N_{G}^P(x,v)\Longrightarrow|\alpha| \leq k |\beta|.$$

\end{proposition}
\begin{proof}
We prove it by contradiction. Suppose that for each $i$ there exist $t_i \in [t_0,t_1]$, $(x_i,v_i) \in G\cap ( B(x_*(t_i),\varepsilon)\times B(\dot{x}_*(t_i), R(t_i))) $ such that
$$(\alpha_i,\beta_i) \in N_{G}^P(x_i,v_i)\Rightarrow|\alpha_i| > i |\beta_i|$$
By normalizing and taking subsequences, we may take $|\alpha_i|=1$, and we may suppose that $(t_i,x_i,v_i)\rightarrow (t,x,v) \in C_*^{\varepsilon, R}$ and $\alpha_i\rightarrow \alpha$, where $\alpha$ is a unit vector. It follows that $\beta_i\rightarrow 0$. Then in the limit we derive
$$(\alpha, 0)\in N^L_G(x,v).$$
Equivalently $\alpha \in D^*\Gamma (x,v)(0)$ with $\alpha\not =0$.
But this contradicts the fact that $$D^*\Gamma (x,v)(0)=\{0\}. $$ The contradiction shows that the bounded slope condition holds for certain constant $k>0$.
\end{proof}
%
%

By Corollary \ref{thm2} and Proposition \ref{propositionnew},
we obtain the following necessary optimality condition.
\begin{thm}\label{thm2new}
Assume that $x_*$  is a local minimum of radius $R(\cdot)$ of the optimal differential inclusion problem $(P_I)$ where  $\Gamma_t(x)\equiv\Gamma(x)$ is independent of $t$. Furthermore suppose that there exists $\delta>0$ such that $R(t) \geq \delta$. If $C_*^{\varepsilon, R}$ is compact and  for all $(t, x,v) \in C_*^{\varepsilon, R}$, one has $D^*\Gamma (x,v)(0)=\{0\}$.
Then the necessary conditions of Theorem \ref{thm2.1} hold as stated, for the same radius $R(\cdot)$.
\end{thm}
\begin{proof} By Proposition \ref{propositionnew}, there exists a constant $k>0$ such that for almost every $t$, for every $(x,v) \in G$ with $x\in B(x_*(t),\varepsilon)$ and $v \in {B}(\dot{x}_*(t),R(t))$, the bounded slope condition holds
$$(\alpha,\beta) \in N_{G}^P(x,v)\Longrightarrow|\alpha| \leq k |\beta|.$$ Since the radius function is bounded away from zero, it follows that the strengthened tempered growth condition holds. Hence all assumptions in Corollary \ref{thm2} holds and so the desired conclusion follows.
\end{proof}

The constraint qualification imposed in Theorem \ref{thm2new} is required to hold for points in a neighborhood of the optimal process $(x_*(t),\dot{x}_*(t))$.
It is natural to ask whether this condition can be imposed only along the optimal process $(x_*(t),\dot{x}_*(t))$.
Inspired by  Clarke and De Pinho \cite[Definition 4.7]{cp}, in order to answer this question we first introduce the following concept.
\begin{defn}
We say that $(t,x_*(t),v)$ is an admissible cluster point of $x_*$ if there exists a sequence $t_i\in [t_0,t_1]$ converging to $t$ and $(x_i,v_i)\in G(t_i)$ such that $\lim x_i=x_*(t)$ and $\lim v_i =\lim \dot{x}_*(t_i) =v$.
\end{defn}
\begin{thm} \label{thm3.3}Assume that $x_*$  is a local minimum of constant radius $R$ of the optimal differential inclusion problem $(P_I)$ where  $\Gamma_t(x)\equiv\Gamma(x)$ is independent of $t$.   Suppose that $\dot{x}_*(t)$  is bounded and that for all  $(t,x_*(t),v)$, admissible cluster point of $x_*$, $D^*\Gamma(x_*(t),v)(0)=\{0\}$. Then the necessary conditions of Theorem \ref{thm2.1} hold as stated, for  some radius $\eta \in (0,R)$.
\end{thm}
\begin{proof}
We first verify that there exists a positive constant $k$ such that for almost all $t\in [t_0,t_1]$ and any $(x,v) \in G$ such that $x\in B(x_*(t),\rho), v\in B(\dot{x}_*(t),\eta)$ with $\rho \in (0,\varepsilon)$ and $\eta\in (0,R)$ sufficiently small,
$$(\alpha,\beta) \in N_{G}^P(x,v)\Longrightarrow|\alpha| \leq k |\beta|.$$
 We prove it by contradiction. Suppose that for each $i$ there exist $t_i \in [t_0,t_1]$, $(x_i,v_i) \in G\cap (B(x_*(t_i),\rho_i)\times B(\dot{x}_*(t_i), \eta_i)) $ with $\rho_i\rightarrow 0, \eta_i\rightarrow 0$ such that
$$(\alpha_i,\beta_i) \in N_{G}^P(x_i,v_i)\Rightarrow|\alpha_i| > i |\beta_i|$$
By normalizing and taking subsequences, we may take $|\alpha_i|=1$ and we may suppose that $(t_i,x_i,v_i)\rightarrow (t,x,v)$ and $\alpha_i\rightarrow \alpha$, where $\alpha$ is a unit vector. It follows that $\beta_i\rightarrow 0$. Since $G$ is closed, we have $(x,v)\in G$. Since
\begin{eqnarray*}
|x_i-x_*(t_i)|< \rho_i &\Longrightarrow & x=x_*(t)\\
|v_i-\dot{x}_*(t_i)|<\eta_i &\Longrightarrow & v=\lim v_i=\lim \dot{x}_*(t_i).
\end{eqnarray*}
It follows that $(t,x_*(t),v)$ is an admissible cluster point of $x_*$.   Then in the limit we derive
$$(\alpha, 0)\in N^L_G(x_*(t),v).$$
Equivalently $\alpha \in D^*\Gamma(x_*(t),v)(0)$ with $\alpha\not =0$ which contradicts the assumption that $D^*\Gamma(x_*(t),v)(0)=\{0\}$.

Hence all assumptions in Corollary \ref{thm2} hold with radius $R(\cdot)$ replaced with $\eta$ and $\varepsilon$ with $\rho$ and so the desired conclusion follows by applying Corollary \ref{thm2}.
\end{proof}

Suppose that $\dot{x}_*(\cdot)$ is a continuous function. Then for all $t\in [t_0,t_1]$, we have $\lim \dot{x}_*(t_i)=\dot{x}_*(t) $ and $\dot{x}_*(\cdot)$ is bounded.  In this case the only admissible cluster point of $x_*$ is $(t,x_*(t), \dot{x}_*(t))$ and hence  the constraint qualification is only needed to be verified along the optimal process $(x_*,\dot{x}_*)$.
\begin{corollary} \label{Cor3.2} Assume that $x_*$  is a local minimum of constant radius $R$ of the optimal differential inclusion problem $(P_I)$ where  $\Gamma_t(x)\equiv\Gamma(x)$ is independent of $t$. Moreover suppose that $\dot{x}_*(\cdot)$ is continuous. Then if $D^*\Gamma(x_*(t),\dot{x}_*(t))(0)=\{0\}$, the necessary optimality conditions of Theorem \ref{thm2.1} hold as stated with some radius $\eta \in (0,R)$.
\end{corollary}

{\bf Remark}:
In the case of free end point, where $E=E_0\times \Re^n$, it follows from the transversality condition that $p(t_1)=0$. Since it follows from the pseudo-Lipschitz condition of $\Gamma_t$ and the Euler inclusion $p$ satisfies
$$|\dot{p}(t)|\leq k(t) |p(t)| \quad a.e. \,t\in [t_0,t_1]$$
if $\lambda_0=0$ then $p(\cdot)=0$ as well. Hence in the case of free end point condition, $\lambda_0$ in the Theorem \ref{thm2.1} and Corollalry  \ref{thm2} and Theorems \ref{thm2new} can be
taken as $1$.

When the initial condition is fixed, that is $x(t_0)=x_0$, the end point is free and $f(x_0,x_1)=f(x_1)$, taking $E=\{x_0\}\times \Re^n$, $\lambda_0$ can be taken as $1$ and the transversality condition becomes
$$-p(t_1) \in \partial^L f(x_*(t_1)).$$

\section{Optimal control with a geometric constraint}

 In this section  we consider   the following optimal control problem with a geometric constraint:
\begin{eqnarray*}
(P_c)~~~~~~\min &&  J(x,u):=\int_{t_0}^{t_1} F(t, x(t), u(t)) dt + f(x(t_0),x(t_1)),\nonumber \\
s.t. && \dot{x}(t)=\phi(t,x(t), u(t)) \,\quad a.e. \,t\in [t_0,t_1],\nonumber\\
&& \Phi(t,x(t), u(t))\in \Omega (t) \quad a.e. \,t\in [t_0,t_1], \label{mconstraint}\\
&& u(t) \in U(t) \quad a.e. \,t\in [t_0,t_1],\nonumber\\
&& (x(t_0), x(t_1))\in E,\nonumber
\end{eqnarray*}
where $F:[t_0,t_1]\times \Re^n\times \Re^m\rightarrow \Re$, $f:\Re^n\times \Re^n\rightarrow \Re$,
$\phi:[t_0,t_1]\times \Re^{n}\times \Re^{m}\rightarrow \Re^n$, $\Phi:[t_0,t_1]\times \Re^{n}\times \Re^m\rightarrow  \Re^d$;
 $U:[t_0,t_1]\rightrightarrows  \Re^m$, $\Omega:[t_0,t_1]\rightrightarrows  \Re^d$ and $E\subset \Re^n \times \Re^n$.   The basic hypotheses on the problem data, in force throughout this section, are the following:
$F, \phi, \Phi$ are Lebesgue measurable in variable $t$ and locally Lipschitz  in variable $(x,u)$; $f$  is locally Lipschitz continuous;
 the set-valued maps $U, \Omega$ are Lebesgue measurable  with closed values.

We say that  $u(t)$ is a {\em control function} if it is
a Lebesgue  measurable function satisfying $u(t) \in U(t)
\mbox{ a.e. }t\in [t_0,t_1]$.
An admissible pair for
$(P_c)$ is a pair of functions $(x(\cdot),u(\cdot))$ on $[t_0,t_1]$ of which $u(\cdot)$ is  a control function and $x(\cdot): [t_0,t_1]\rightarrow \Re^n$
is an arc that satisfies all the constraints.
Let $R(\cdot):[t_0,t_1]\rightarrow (0,+\infty]$ be a given measurable radius function.
As in \cite{cr}, a {\em $W^{1,1}$ local minimum of radius $R(\cdot)$} to problem $(P_c)$ is an admissible pair $(x_{*},u_{*})$ that minimizes the value of the cost
function $J(x, u)$ over all admissible pairs $(x,u)$ which satisfies
$$| \dot{x}(t)-\dot{x}_{*}(t)|\leq R(t) \mbox{ a.e.,} \int_{t_0}^{t_1}|\dot{x}(t)-\dot{x}_{*}(t)|dt\leq \varepsilon,  \| x-x_{*}\|_{\infty}\leq \varepsilon.$$

Similarly as in \cite{cp}, we formulate our problem as a general optimal differential inclusion problem and apply the theory from Clarke \cite{cr} and the results obtained in Section 3.
It is worth to note,  however, that our $W^{1,1}$ local minimum  concept is slightly different from those in \cite{cp} where a  {\em $W^{1,1}$ local minimum of radius $R(\cdot)$} means that $(x_{*},u_{*})$  minimizes  $J(x, u)$ over all admissible pairs $(x,u)$ which satisfies both $| x(t)-x_{*}(t)|\leq \varepsilon$, $|u(t)-u_{*}(t)|\leq R(t) \mbox{ a.e.}$ and
$\int_{t_0}^{t_1}|\dot{x}(t)-\dot{x}_*(t)|dt\leq \varepsilon$. Nevertheless the technique we use can be also used to the solution concepts in \cite{cp} to obtain similar results. Although the technique used in our work is similar to that in \cite{cp}, which  transforms an optimal control problem into an optimal differential inclusion problem, the obtained optimal differential inclusion problem in our work is much easier to deal with. Therefore the proof is much more concise compared with the proof of Theorem 2.1 in \cite{cp}.
Define the perturbed set-valued map  $M_{t}: \Re^d\rightrightarrows \Re^{n+m}$ by
\begin{equation}
M_{t}(y):=\{(x,u)\in \Re^n\times U(t): \Phi(t,x,u)+y\in \Omega(t)\}.\label{calmness}
\end{equation}
The following theorem asserts necessary optimality conditions under  the pseudo-Lipschitz continuity of radius $R(\cdot)$ of the set-valued map (\ref{dynamics}), the strengthened tempered growth condition and the calmness of the set-valued map $M_t$ defined as in (\ref{calmness}).
\begin{thm}\label{thm4.1}Let  $(x_*,u_*)$ be a $W^{1,1}$ local minimum of radius $R(\cdot)$ for $(P_c)$.
 In additions to the basic hypothesis,
 assume that  there exist $\varepsilon>0$ and an uniformly positive integrable function $k$ such that for almost every $t\in [t_0,t_1]$ the following is true:
 for every $x,x'\in B(x_{*}(t),\varepsilon)$, and every
 $u\in U(t)$ with $\Phi(t,x,u)\in \Omega(t)$ and $$|(\phi(t,x,u),F(t,x,u))-(\phi(t,x_*(t),u_*(t)),F(t,x_*(t),u_*(t)))|\leq R(t),$$
there exists
$u'\in U(t)$ with $\Phi(t,x',u')\in \Omega(t)$
such that
 $$|(\phi(t,x,u),F(t,x,u))-(\phi(t,x',u'),F(t,x',u'))|\leq k(t)|x-x'|.$$
Suppose further  that
$R(t)\geq \delta k(t) \ a.e. t\in [t_0,t_1]$  for some $\delta >0$.
Moreover suppose that  the set-valued map $M_{t}$
is calm at $(0,x_*(t),u_*(t))$ for almost every $t\in [t_0,t_1]$.
Then there exist an arc $p$ and   a number $\lambda _{0}$ in $\{0,1\}$,
satisfying the nontriviality condition
$(\lambda _{0},p(t))\neq0, \forall t\in[t_0,t_1]$,
 the transversality condition
$$(p(t_0),-p(t_1)) \in \lambda_0 \partial^Lf(x_*(t_0),x_*(t_1))+N_E^{L}(x_*(t_0),x_*(t_1)),$$
and
the Euler adjoint inclusion for almost every $t$:
\begin{eqnarray}
(\dot{p}(t), 0) \in
&& \hspace{-0.5cm}\partial^C \{\langle -p(t), \phi(t,\cdot,\cdot)\rangle+\lambda_0F(t,\cdot,\cdot)\} ( x_*(t),u_*(t))+\{0\}\times N^C_{U(t)}(u_*(t))\nonumber \\
&& +co\{ \partial^L\langle  \lambda,  \Phi(t,\cdot,\cdot)\rangle( x_*(t),u_*(t)): \lambda\in N_{\Omega(t)}^{L}(\Phi(t,x_*(t),u_*(t))\}, \label{EulerN}
\end{eqnarray}
as well as the Weierstrass condition of radius $R(\cdot)$ for almost every $t$:
\begin{eqnarray*}
&&\Phi(t,x_*(t), u(t))\in \Omega (t), \, \,  |\phi(t,x_*(t), u)-\phi(t, x_*(t),u_*(t))|< R(t)\Longrightarrow\\
&&\langle p(t),\phi(t,x_*(t),u) \rangle-\lambda_{0}F(t,x_*(t),u)\leq \langle p(t), \phi (t,x_*(t),u_*(t)) \rangle -\lambda_{0}F(t,x_*(t),u_*(t)).
\end{eqnarray*}
\end{thm}
\begin{proof} We shall first prove the theorem for the case  $F\equiv 0$.  For each $ t\in [t_0,t_1]$ define a set-valued map
\begin{equation}
\Gamma_t(x):=\{\phi(t,x,u):u\in U(t), \Phi(t,x,u) \in \Omega(t)\}.\label{dynamics}
\end{equation}


Now we prove the optimal control problem $(P_c)$ and the optimal differential inclusion problem $(P_I)$ with the set-valued map $\Gamma_t(x)$ defined as in (\ref{dynamics}) are equivalent under the basic hypotheses.
Let the sets of the admissible arcs of $(P_c)$ and $(P_I)$ be ${\cal A}$, ${\cal B}$ respectively. It is obvious that ${\cal A}\subset {\cal B}$
so we only need to prove ${\cal B}\subset {\cal A}$.
Assume  that $x(\cdot)\in {\cal B}$. Then there exists $u\in U(t)$ such that $$\dot{x}(t)=\phi(t,x(t),u), \  \Phi(t,x(t),u)\in \Omega(t) \qquad a.e. t\in [t_0,t_1].$$
Since $\dot{x}(t)$ can be viewed as the limit of (absolutely) continuous functions sequence $\{n(x(t+\frac{1}{n})-x(t))\}$, $\dot{x}(t)$ is measurable.  Therefore the mappings $\Phi(t,x(t),u)$ and $\dot{x}(t)-\phi(t,x(t), u)$ are measurable in $t$ and continuous in $u$.  Moreover since $\Omega$ is closed valued and measurable,  by virtue of   \cite[Example 14.15(b)]{rw}, the set-valued map
$$I(t):=\{u\in \Re^m:\Phi(t,x(t),u) \in \Omega(t), \dot{x}(t)-\phi(t,x(t), u)=0\}$$
is measurable.
By \cite[Exercise 3.5.2(h)]{clsw}, the set-valued map $I(\cdot)\cap U(\cdot)$ is also a closed-valued measurable set-valued map.
It follows from the Filippov's Lemma (see e.g. \cite[Problem 3.7.20, Page 174]{clsw}) that there exists  a measurable function $u(\cdot):[t_0,t_1]\rightarrow \Re^{m}$
such that $u(t) \in U(t), \Phi(t,x(t),u(t)) \in \Omega(t)$ and $\dot{x}(t)=\phi(t,x(t), u(t))$ for almost all $ t\in [t_0,t_1]$, which proves ${\cal B}\subset {\cal A}$.
Hence the arc $x_*(t)$ is a $W^{1,1}$ local solution of  the optimal  differential inclusion problem $(P_I)$.

By the basic hypotheses, the set-valued maps $U,\Omega$ are Lebesgue measurable, the mappings $\phi(t,x,u)$ and  $\Phi(t,x,u)$ are measurable in variable $t$ and locally Lipschitz in variables $(x,u)$. It follows that  the map $(t,x)\rightarrow \Gamma_{t}(x)$ is ${\cal L}\times {\cal B}$-measurable and the
 graph of $\Gamma_t$  defined as
$$G(t):=\{(x,\phi(t,x,u)): u\in U(t), \Phi(t,x,u)\in \Omega(t)\}$$
is closed for each $t$. With the assumptions and the definition of $\Gamma_t$, it is easy to  verify the pseudo-Lipschitz condition of radius $R(\cdot)$ for $\Gamma_t$.
 By the assumption, we also get that $\frac{R(t)}{k(t)}$ is bounded away from 0,
hence the strengthened tempered growth condition holds. Hence all assumptions for Theorem   \ref{thm2.1} are verified.

By Theorem  \ref{thm2.1},
there exists an arc $p$ and a number $\lambda_0 $ in  $\{0,1\}$ satisfying the nontriviality condition (\ref{nontrivi}),
 the transversality condition (\ref{transv}), the Euler inclusion (\ref{Euler})
and the Weierstrass condition (\ref{Weires}).
Note that the nontriviality condition and the transversality condition in this theorem are exactly  (\ref{nontrivi}) and  (\ref{transv}) respectively and the Weierstrass condition of this theorem can be easily  derived by the Weierstrass condition (\ref{Weires}).
Hence at this point, we have an arc $p$ and a number $\lambda _{0}$ in $\{0,1\}$ satisfying all the required conclusions except that the Euler adjoint inclusion (\ref{EulerN}). We proceed to derive (\ref{EulerN}) from the Euler inclusion (\ref{Euler}). Note that the Euler inclusion
(\ref{Euler}) is equivalent to
\begin{equation}
( \dot{p}(t),0) \in co\left\{(\omega,0): (\omega,p(t))\in N_{G(t)}^{L}(x_*(t),\dot{x}_{*}(t))\right\},\, a.e.\, t\in [t_0,t_1].
\label{EulerNN(}
\end{equation}
We suppose that $t$ is a point where the Euler inclusion holds, the calmness of $M_t$ holds and $ \dot{x}_*(t)=\phi(t,x_*(t),u_*(t))).$
Let $\omega$ be any point satisfying $$(\omega,p(t))\in N_{G(t)}^{L}(x_*(t),\phi(t,x_*(t),u_*(t))).$$ Then there exist sequences
$$(x_{i},\phi(t,x_{i},u_{i}))\xrightarrow{G(t)}(x_*(t),\phi(t,x_*(t),u_*(t)))$$ and
 $(\omega_{i},p_{i})\rightarrow(\omega,p(t))$ such that $(\omega_{i},p_{i})\in N_{G(t)}^{P}(x_i,\phi(t,x_i,u_i))$. By the definition of proximal normal vector, for each
$i$ there exists $\sigma_{i}\geq 0$ such that if $(x,u)\in M_{t}(0)$ we have
$$\langle(\omega_{i},p_{i}),(x,\phi(t,x,u))-(x_{i},\phi(t,x_{i},u_{i}))\rangle\leq \sigma_{i}|(x,\phi(t,x,u))-(x_{i},\phi(t,x_{i},u_{i}))|^{2}.$$
It follows that the function
$$\Lambda(x,u):=\langle-(\omega_{i},p_{i}),(x,\phi(t,x,u))\rangle+ \sigma_{i}|(x,\phi(t,x,u))-(x_{i},\phi(t,x_{i},u_{i}))|^{2}$$
has a  minimum at $(x_{i},u_{i})\in M_{t}(0)$.
Therefore by the Lipschitz continuity of $\phi$ in variable $(x,u)$ and the closedness of the set $M_t(0)$, we have
$$(0,0)\in -(w_i,0)+\partial^L\{\langle -p_i,\phi(t,\cdot,\cdot)\rangle\}(x_{i},u_{i})+N_{M_{t}(0)}^{L}(x_{i},u_{i}).$$
Passing to the limit, by virtue of the Lipschitz continuity of $\Omega$ and the closedness of the limiting normal cone mapping, we get  that
$$(w,0)\in \partial^L\{\langle -p(t),\phi(t,\cdot,\cdot)\rangle\}(x_*(t),u_*(t))+N_{M_{t}(0)}^{L}(x_*(t),u_*(t)).$$
Since the set-valued map $M_{t}$ is calm at $(0,x_*(t),u_*(t))$ and $\Phi$ is locally Lipschitz in variables $(x,u)$, we get that, by \cite[Proposition 3.4]{adi} and \cite[Proposition 2.11]{morduk}
\begin{eqnarray*}
  N_{M_{t}(0)}^{L}(x_*(t),u_*(t))\subset \{ \partial^L\hspace{-0.7cm}&&\langle \lambda, \Phi(t,\cdot,\cdot)\rangle(x_*(t),u_*(t)):\\
&&
\lambda\in N_{\Omega(t)}^{L}(\Phi(t,x_*(t),u_*(t)))\}+N_{\Re^n\times U(t)}^{L}(x_*(t),u_*(t)).
\end{eqnarray*}
By (\ref{EulerNN(}) and the fact that
 the closed convex
hull of $\partial^L, N^{L}$ is $\partial^C,N^{C}$ respectively, we get
\begin{eqnarray*}
(\dot{p}(t),0)\in \partial^C\hspace{-0.7cm}&&\{\langle -p(t),\phi(t,\cdot,\cdot)\rangle\}(x_*(t),u_*(t))+\{0\}\times N_{U(t)}^{C}(u_*(t))+\\
&&co\{ \partial^L\langle \lambda, \Phi(t,\cdot,\cdot)\rangle(x_*(t),u_*(t)):
\lambda\in N_{\Omega(t)}^{L}(\Phi(t,x_*(t),u_*(t)))\},
\end{eqnarray*}
which completes the proof in the case $F\equiv0$. However, in the case  where a nonzero $F$ is present, the added integral cost can be absorbed into the inclusion by
the familiar  technique of state augmentation: we introduce an additional one-dimensional state variable satisfying $\dot{y}(t)=F(t,x(t),u(t))$
together with $y(t_0)=0$. The problem now amounts to minimizing $f(x(t_0),x(t_{1}))+y(t_{1})$ subject to the augmented differential equations.
Then the conclusions of the theorem follow immediately under the assumptions of the theorem by the same arguments as above.
\end{proof}

In the proof of Theorem \ref{thm4.1}, we have shown that  the optimal control problem $(P_c)$ is equivalent to the  optimal differential inclusion problem $(P_I)$ where the set-valued map $\Gamma_t(x)$ is defined as in (\ref{dynamics}) and have translated the assumptions from the control problems to the differential inclusion problems so that the assumptions such as the measurability and the pseudo-Lipschitz continuity  hold for the  set-valued map $\Gamma_t(x)$. The interested reader is referred to \cite{pinho7} for more properties that can be translated from the mixed control problem to the differential inclusion problems.

It is not easy to check the integrability of the function $k$ in Theorem \ref{thm4.1}.
 This difficulty can be dealt with if the function $k$ is constant. We consider the case where the induced differential inclusion problem  is autonomous, i.e.,
 $ F,\phi, \Phi, U, \Omega$ are all independent of $t$, and apply the corresponding results from Section 3. From now on, we suppress the notation $t$ when
 $ F,\phi, \Phi, U, \Omega$ are independent of $t$.
For this purpose we give the following definitions:
For any given $\varepsilon>0$ and a given radius function $R(t)$, define
$$S_*^{\varepsilon,R}(t):=\{(x,u)\in \bar{B}(x_*(t),\varepsilon)\times U:  \Phi(x,u) \in \Omega , |\phi(x,u)-\dot{x}_*(t)|\leq R(t)\}.$$
Let
$$C_*^{\varepsilon,R}=cl\{(t,x,v)\in [t_0,t_1]\times \Re^n\times \Re^n: v=\phi(x,u),(x,u)\in S_*^{\varepsilon,R}(t)\},$$
where $cl$ denotes the closure.

\begin{thm}\label{thm4.2new}  Let  $(x_*,u_*)$ be a $W^{1,1}$ local minimum of radius $R(\cdot)$ for $(P_c)$ where    $U, F,$ $\phi, \Phi,\Omega$ are all independent of $t$.  Suppose that there exists $\delta>0$ such that $R(t)\geq \delta$.
 Suppose that $C_*^{\varepsilon, R}$ is compact and that for  all  $(t,x,u)$ with $(t,x, \phi(x,u))\in C_*^{\varepsilon,R}$, the WBCQ holds:
\begin{eqnarray*}
\left \{ \begin{array}{l}  (\alpha,0)\in \partial^L\langle \lambda(t), \Phi(\cdot,\cdot)\rangle(x,u)+\{0\}\times N_{U}^{L}(u)\\
\lambda(t)\in N_{\Omega}^{L}(\Phi(x,u))\end{array} \right. \Longrightarrow \alpha=0
\end{eqnarray*}
and the mapping $M$ defined as in (\ref{perturbedmap}) is calm at $(0, x,u)$.
Then the necessary optimality conditions of Theorem \ref{thm4.1} hold as stated with the same radius $R(\cdot)$.
\end{thm}
\begin{proof} Similarly as in the proof of Theorem \ref{thm4.1}, it suffices to  prove the case $F\equiv 0$. Let
\begin{equation}
\Gamma(x):=\{\phi(x,u): u\in U, \Phi(x,u) \in \Omega\}.\label{dynamic}
 \end{equation} Then as in the proof of Theorem \ref{thm4.1}, problem $(P_c)$ is equivalent to problem $(P_I)$.
By   Proposition \ref{prop2.2}, for all  $(t,x,u)$ with $(t,x, \phi(x,u))\in C_*^{\varepsilon,R}$, the WBCQ plus the calmness of $M$ implies that $\Gamma$ is pseudo-Lipschitz around $(x,\phi(x,u))$. By Proposition \ref{prop2.0}, we get that  $D^* \Gamma(x,\phi(x, u))(0)=\{0\}$. Using the similar proof as in Theorem \ref{thm4.1}, we obtain the result by applying Theorem \ref{thm2new}.
\end{proof}

\begin{defn}\label{Defn4.2}
We say that $(t,x_*(t),v)$ is an admissible cluster point of $x_*$ if there exists a sequence $t_i\in [t_0,t_1]$ converging to $t$ and $u_i\in U(t_i), \Phi(t_i,x_i,u_i)\in  \Omega(t_i)$ such that $\lim x_i=x_*(t)$, and $v=\lim  \phi(t_i,x_i,u_i)=\lim \dot{x}_*(t_i) $.
\end{defn}
\begin{thm}\label{thm4.3} Let  $(x_*,u_*)$ be a $W^{1,1}$ local minimum of constant radius $R$ for $(P_c)$ where  $U, F,\phi,$ $ \Phi,\Omega$ are all independent of $t$.  Suppose that  $\dot{x}_*(t)$  is bounded.
Assume that for every $(x_*(t),u)$ such that $(t,x^*(t), \phi(x^*(t), u))$ is an admissible cluster point of $x_*$,  the WBCQ holds:
\begin{eqnarray*}
\left \{ \begin{array}{l}  (\alpha,0)\in \partial^L\langle \lambda(t), \Phi(\cdot,\cdot)\rangle(x_*(t),u)+\{0\}\times N_{U}^{L}(u)\\
\lambda(t)\in N_{\Omega}^{L}(\Phi(x_*(t),u))\end{array} \right. \Longrightarrow \alpha=0
\end{eqnarray*}
and the map $M$ is calm at $(0, x_*(t),u)$.
Then the necessary optimality conditions of Theorem \ref{thm4.1} hold as stated with some radius $\eta \in (0,R)$. Moreover if  $\dot{x}_*(\cdot)$ is continuous, then the WBCQ and the calmness condition are only required to hold along  $(x_*(t),u_*(t))$.
\end{thm}
\begin{proof} Similarly as in the proof of Theorem \ref{thm4.1}, it suffices to  prove the case $F\equiv 0$. Let
$\Gamma(x)$ be the set-valued map defined as in (\ref{dynamic}). Then as in the proof of Theorem \ref{thm4.1}, problem $(P_c)$ is equivalent to problem $(P_I)$.
By   Proposition \ref{prop2.2},  the WBCQ plus the calmness of $M$ implies that $\Gamma$ is pseudo-Lipschitz around every $(t,\phi(x_*(t),u))$, an admissible cluster point of $x_*$. By Proposition \ref{prop2.0}, we get that  $D^* \Gamma(x_*(t),\phi(x_*(t),u))(0)=\{0\}$. Using the similar proof as Theorem \ref{thm4.1}, we obtain the result by applying Theorem \ref{thm3.3} and Corollary \ref{Cor3.2}.
\end{proof}


The Euler adjoint inclusion (\ref{EulerN}) in Theorem \ref{thm4.1} is in an implicit form. In the case where $\Phi$ is smooth, it is easy to see that the Euler adjoint inclusion (\ref{EulerN}) takes the form
\begin{eqnarray*}
(\dot{p}(t), 0) \in
&& \hspace{-0.5cm}\partial^C \{\langle -p(t), \phi(t,\cdot,\cdot)\rangle+\lambda_0F(t,\cdot,\cdot)\} ( x_*(t),u_*(t))+\{0\}\times N^C_{U(t)}(u_*(t))\nonumber \\
&& +\{\nabla \Phi(t, x_*(t),u_*(t))^T \lambda: \lambda\in N_{\Omega(t)}^{C}(\Phi(t,x_*(t),u_*(t))\}\nonumber
\end{eqnarray*}
and by using the measurable selection theorem, one can find a measurable multiplier $\lambda(t)\in N_{\Omega(t)}^{C}(\Phi(t,x_*(t),u_*(t))$ such that the Euler adjoint inclusion in an explicit multiplier form \begin{eqnarray*}
(\dot{p}(t), 0) \in
&& \hspace{-0.5cm}\partial^C \{\langle -p(t), \phi(t,\cdot,\cdot)\rangle+\lambda_0F(t,\cdot,\cdot)\} ( x_*(t),u_*(t))+\{0\}\times N^C_{U(t)}(u_*(t))\nonumber \\
&& +\nabla \Phi(t, x_*(t),u_*(t))^T \lambda(t)
\end{eqnarray*}
holds.
Note that in Clarke and de Pinho \cite[Theorems 4.3 and 4.8]{cp},  the Euler adjoint inclusion in an explicit multiplier form were also derived under the assumption that the constraint function $\Phi$ is smooth. In the general nonsmooth case, it is difficult to derive the Euler inclusion in the explicit multiplier form. The difficult lies in that in general $\alpha C+\beta C \subset (\alpha+\beta) C$ may not hold if $\alpha ,\beta$ are real numbers (not all positive) and $C$ is a convex set. However if $\lambda$ in the Euler inclusion (\ref{EulerN}) are all nonnegative, then this difficulty does not exist as  shown in  the following theorem.
\begin{thm}\label{New}  In additions to  the assumptions of Theorems \ref{thm4.1}, \ref{thm4.2new} and \ref{thm4.3},
 suppose that  the mapping $\Phi=(\Phi_1,\Phi_2)$, the set-valued map $\Omega(t)=\Omega_1(t)\times \Omega_2(t)$, $d=d_1+d_2$  where $\Phi_1$ is smooth and $N_{\Omega_2(t)}^{L}(\Phi_{2}(t,x_*(t),u_*(t))\subset \Re_{+}^{d_2}$.
Then the Euler adjoint inclusion can be replaced by the one in the explicit multiplier form, i.e., there exists
 a measurable function $\lambda:[t_0,t_1]\rightarrow \Re^d$ with
$\lambda(t)\in N_{\Omega(t)}^{C}(\Phi(t,x_*(t),u_*(t))$ for almost every $t\in [t_0,t_1]$
satisfying
\begin{eqnarray*}
(\dot{p}(t),0)\in &&\hspace{-0.5cm}\partial^C \{\langle -p(t), \phi(t,\cdot,\cdot)\rangle+\lambda_0F(t,\cdot,\cdot)\} ( x_*(t),u_*(t))+\\
&& \partial^C\Phi(t,x_*(t),u_*(t))^T\lambda(t)+\{0\}\times N_{U(t)}^{C}(u_*(t)).
\end{eqnarray*}
\end{thm}

\begin{proof}The proof for the smooth part is obvious as shown in the discussion before the theorem. For simplicity we assume that there is no smooth part.  Then it suffices to  prove
\begin{eqnarray*}
 co\{ \partial^L\hspace{-0.7cm}&&\langle \lambda, \Phi(t,\cdot,\cdot)\rangle(x_*(t),u_*(t)):\lambda\in N_{\Omega(t)}^{L}(\Phi(t,x_*(t),u_*(t)))\}\\
 &&\subset \{\partial^C\Phi(t,x_*(t),u_*(t))^T\lambda:
\lambda\in N_{\Omega(t)}^{C}(\Phi(t,x_*(t),u_*(t)))\}.
\end{eqnarray*}
By the Caratheodory's theorem, any $$\zeta\in co\{ \partial^L\langle \lambda, \Phi(t,\cdot,\cdot)\rangle(x_*(t),u_*(t)):
\lambda\in N_{\Omega(t)}^{L}(\Phi(t,x_*(t),u_*(t)))\}$$ can be expressed as a convex combination
of $\mbox{n+m+1}$ points at most, namely, there exist  $\varrho_{i}\geq 0, \sum\limits_{i=1}^{n+m+1}\varrho_{i}=1$ and $\lambda_{i}\in N_{\Omega(t)}^{L}(\Phi(t,x_*(t),u_*(t)))$ such that
$$\zeta\in \sum\limits_{i=1}^{n+m+1}\varrho_{i}\partial^L\langle \lambda_{i}, \Phi(t,\cdot,\cdot)\rangle(x_*(t),u_*(t))\subset \sum\limits_{i=1}^{n+m+1}\partial^C\langle \varrho_{i}\lambda_{i}, \Phi(t,\cdot,\cdot)\rangle(x_*(t),u_*(t)).$$
By \cite[Exercise 13.32]{clar}, we get $\zeta\in \sum\limits_{i=1}^{n+m+1} \partial^C\Phi(t,x_*(t),u_*(t))^T(\varrho_{i}\lambda_{i})$.
Since $\lambda_{i}\geq 0$, we get $\zeta \in   \partial^C\Phi(t,x_*(t),u_*(t))^T (\sum\limits_{i=1}^{n+m+1}\varrho_{i}\lambda_{i})$ by \cite[Theorem 3.2]{rwf} where $\sum\limits_{i=1}^{n+m+1}\varrho_{i}\lambda_{i}\in N_{\Omega(t)}^{C}(\Phi(t,x_*(t),u_*(t)))$.

By \cite[Theorem 14.26]{rw}, one can easily get the measurability of the mapping $t\rightarrow N_{\Omega(t)}^{C}(\Phi(t,x_*(t),u_*(t)))$.
Hence we get that, by the measurable selection theorem (see e.g. \cite[Theorem 3.1.1]{c}), there exists a measurable function $\lambda(t)\in N_{\Omega(t)}^{C}(\Phi(t,x_*(t),u_*(t)))$ such that
the Euler adjoint inclusion in the explicit multiplier form holds.
\end{proof}

\section{Optimal control with equality and inequality constraints}
In this section we  consider a special case of problem $(P_c)$ where the mixed constraints are of equality/inequality type defined as follow:
\begin{eqnarray*}
(P_m)~~~~~~\min &&  J(x,u):=\int_{t_0}^{t_1} F( x(t), u(t)) dt + f(x(t_0),x(t_1)),\\
s.t. && \dot{x}(t)=\phi(x(t), u(t)), \,\,\, a.e. \,t\in [t_0,t_1],\\
&& g(x(t),u(t))\leq 0,\,\, \mbox{a.e.} t\in [t_0,t_1],\\
&&h(x(t),u(t))=0,\,\, \mbox{a.e.} t\in [t_0,t_1],\\
&& u(t) \in U, \,\, a.e. \,t\in [t_0,t_1],\\
&& (x(t_0),x(t_1)) \in E,
\end{eqnarray*}
where  $F:\Re^n\times \Re^m\rightarrow \Re$, $f:\Re^n\times \Re^n\rightarrow \Re$,
$\phi: \Re^{n}\times \Re^{m}\rightarrow \Re^n$, $g:\Re^n\times \Re^m\rightarrow \Re^{l}$  are locally Lipschitz, $h:\Re^n\times \Re^m\rightarrow \Re^{s}$  is strictly differentiable and
 $U\subset  \Re^m$,  $E\subset \Re^n\times \Re^n$ are closed.
Let $R:[t_0,t_1]\rightarrow (0,+\infty]$ be a given measurable radius function. An {\em $W^{1,1}$ local minimum of radius $R(\cdot)$} to problem $(P_m)$ is an admissible pair $(x_{*},u_{*})$ that minimizes the value of the cost
function $J(x, u)$ over all admissible pairs $(x,u)$ which satisfies
$$| \dot{x}(t)-\dot{x}_{*}(t)|\leq R(t) \mbox{ a.e.,} \int_{t_0}^{t_1}|\dot{x}(t)-\dot{x}_{*}(t)|dt\leq \varepsilon,  \| x-x_{*}\|_{\infty}\leq \varepsilon.$$
  When $R=+\infty$, this has been referred to as a
{\em $W^{1,1}$ local minimum } to problem $(P_m)$.
Set $\Phi:=(g,h)^{T}$, $\Omega:=\Omega(t)=\Re^{l}_{-}\times \{0\}$ and
\begin{equation}
M(y_1,y_2):=\{(x,u)\in \Re^n\times U: g(x,u)+y_1\leq 0,\, h(x,u)+y_2=0\}. \label{neweqn}
\end{equation}
We now specialize some results in Section 4  to the problem $(P_m)$.
Note that now
$$
C_*^{\varepsilon,R} = cl \left \{
(t,x,\phi(x,u))\in [t_0,t_1]\times \Re^n\times \Re^n: \begin{array}{ll}
 (x,u)\in \bar{B}(x_*(t),\varepsilon)\times U,  g(x,u)\leq 0, \\
h(x,u)=0 , |\phi(x,u)-\dot{x}_*(t)|\leq R(t)
 \end{array} \right \}.
$$
The following theorem results from  Theorems \ref{thm4.2new}  and   \ref{New}.
\begin{thm}\label{thm5.2}  Let $(x_*,u_*)$ be a $W^{1,1}$ local minimum of radius $R(\cdot)$ for $(P_m)$.
Suppose that there exists $\delta>0$ such that $R(t)\geq \delta$. Suppose that $C_*^{\varepsilon, R}$ is compact and that for all $(t,x,u)$ with $(t,x,\phi(x,u))\in C_*^{\varepsilon, R}$,
 the  WBCQ holds:
\begin{eqnarray*}
\left \{ \begin{array}{l}  (\alpha,0)\in \partial^L \{\langle  \lambda, g(\cdot,\cdot)\rangle\}(x,u)+\nabla h(x,u)^T \varpi+\{0\}\times N_{U}^L(u)\\
\lambda \in \Re^{l}_{+},\langle\lambda,g(x,u)\rangle=0,\varpi \in \Re^s\end{array} \right. \Longrightarrow \alpha=0
\end{eqnarray*}
and the set-valued map $M$ defined as in (\ref{neweqn}) is calm at $(0,x,u)$.
Then there exist an arc $p$ and a number $\lambda _{0}$ in $\{0,1\}$ satisfying the nontriviality condition $(\lambda _{0},p(t))\neq0, \forall t\in[t_0,t_1]$
and  measurable functions $\varpi:[t_0,t_1]\rightarrow \Re^{s}$, $\lambda:[t_0,t_1]\rightarrow \Re^{l}_{+}$ with
$$\langle\lambda(t), g( x_*(t),u_*(t))\rangle=0 \mbox{ a.e. }t\in [t_0,t_1]$$
such that $p$ satisfies the transversality condition
$$(p(t_0),-p(t_1)) \in \lambda_0 \partial^Lf(x_*(t_0),x_*(t_1))+N_E^{L}(x_*(t_0),x_*(t_1)),$$
the Euler adjoint inclusion  in the explicit multiplier form for almost every $t$:
\begin{eqnarray*}
(\dot{p}(t), 0) \in
&& \hspace{-0.5cm}-\partial^C  \phi( x_*(t),u_*(t))^Tp(t)+\lambda_{0}\partial^C F( x_*(t),u_*(t))+\\
&& \partial^Cg( x_*(t),u_*(t))^T\lambda(t) +  \nabla h( x_*(t),u_*(t))^T \varpi(t)+\{0\}\times N^C_{U}(u_*(t)),
\end{eqnarray*}
as well as the Weierstrass condition of radius $R(\cdot)$ for almost every $t$:
\begin{eqnarray*}
&&g(x_*(t),u)\leq 0, \, \,
h(x_*(t),u)=0,\,\, |\phi(x_*(t), u)-\phi( x_*(t),u_*(t))|< R(t)\Longrightarrow\\
&& \langle p(t),\phi(x_*(t),u) \rangle-\lambda_{0}F(x_*(t),u)\leq \langle p(t), \phi (x_*(t),u_*(t)) \rangle -\lambda_{0}F(x_*(t),u_*(t)).
\end{eqnarray*}
\end{thm}
When the radius $R\equiv+\infty$, denote by
$$
C_*^{\varepsilon} = cl \left \{
(t,x,\phi(x,u))\in [t_0,t_1]\times \Re^n\times \Re^n: \begin{array}{ll}
 (x,u)\in \bar{B}(x_*(t),\varepsilon)\times U, \\
  g(x,u)\leq 0,
h(x,u)=0
 \end{array} \right \}.
$$ The following global result follows from Theorem \ref{thm5.2}.
\begin{thm}\label{thm5.3}  Let $(x_*,u_*)$ be a $W^{1,1}$ local minimum for $(P_m)$.
Suppose that the set $C_*^{\varepsilon}$ is compact and that for all $(t,x,u)$ with $(t,x,\phi(x,u))\in C_*^{\varepsilon}$,
 the  WBCQ holds:
\begin{eqnarray*}
\left \{ \begin{array}{l}  (\alpha,0)\in \partial^L\{\langle  \lambda, g(\cdot,\cdot)\rangle\}(x,u)+\nabla h(x,u)^T \varpi+\{0\}\times N_{U}^L(u)\\
\lambda \in \Re^{l}_{+},\langle\lambda,g(x,u)\rangle=0, \varpi \in \Re^s\end{array} \right. \Longrightarrow \alpha=0
\end{eqnarray*}
and the set-valued mapping $M$ defined as in (\ref{neweqn}) is calm at $(0,x,u)$.
Then the necessary optimality conditions of Theorem \ref{thm5.2} hold with
the global  Weierstrass condition:  for almost every $t$:
\begin{eqnarray*}
g(x_*(t),u)\leq 0, \, \,
h(x_*(t),u)=0\Longrightarrow \hspace{-0.6cm}&&\langle p(t),\phi(x_*(t),u) \rangle-\lambda_{0}F(x_*(t),u)\\
&&\leq \langle p(t), \phi (x_*(t),u_*(t)) \rangle -\lambda_{0}F(x_*(t),u_*(t)).
\end{eqnarray*}
\end{thm}
Recall by Definition \ref{Defn4.2} that
$(t,x_*(t),v)$ is an admissible cluster point of $x_*$ if there exists a sequence $t_i\in [t_0,t_1]$ converging to $t$ and $u_i\in U, g(x_i,u_i)\leq 0, h(x_i,u_i)=0$ such that $\lim x_i=x_*(t)$, and $v=\lim  \phi(x_i,u_i)=\lim \dot{x}_*(t_i) $.
By Theorems \ref{thm4.3} and \ref{New}  we have the following local result when the constraint qualification holds only at admissible cluster points.
\begin{thm}\label{thm5.3new} Let  $(x_*,u_*)$ be a $W^{1,1}$ local minimum of constant radius $R$ for $(P_m)$.
Assume that $\dot{x}_*(t)$  is bounded and  that for every $(x_*(t),u)$ such that $(t,x_*(t), \phi(x_*(t), u))$ is an admissible cluster point of $x_*$,  the WBCQ holds at $(x_*(t),u)$ and the  set-valued map $M$ defined as in (\ref{neweqn})  is calm at $(0,x,u)$. Then the necessary optimality conditions in Theorem \ref{thm5.2} hold as stated, with the some radius $\eta\in (0,R)$. Moreover if $\dot{x}_*(\cdot)$ is continuous, then the WBCQ and the calmness condition are only required to hold at $(x_*(t),u_*(t))$.
\end{thm}

\section{Conclusions}
The main purpose of this paper is to derive necessary optimality conditions for the optimal control problem with  nonsmooth mixed state and control constraint under the WBCQ plus the calmness condition on the set-valued map $M$ defined as in (\ref{perturbedmap}). Such conditions are weaker than the  MFC and the calibrated constraint qualification.  In order to reach this goal, we first derived some preliminary results in variational analysis. The main result is a sufficient condition for the pseudo-Lipschitz continuity of a set-valued map $\Gamma$ defined as in (\ref{set-valued map}) in terms of the WBCQ plus the calmness condition of $M$. We then derive some new optimality conditions  for optimal   differential inclusion problems.  Finally we apply these optimality conditions to optimal control problems with a geometric constraint to obtain the necessary optimality condition under the WBCQ plus the calmness condition of $M$. Moreover  we derive the Euler inclusion in the explicit multiplier form under the condition that either the multipliers are nonnegative or the constraint functions are nonsmooth. Applying the results to  the case of the classical  equality and inequality constraints we  illustrate the results  obtained.

\section*{Acknowledgments} We thank the anonymous reviewers of this paper for insightful comments that helped us to improve the presentation of the manuscript.

\end{document}